\documentclass[11pt, a4paper, reqno, english]{amsart}
\usepackage{amsmath} 
\usepackage{amsthm} 
\usepackage{amssymb} 
\usepackage[dvipsnames]{xcolor}
\usepackage{amscd} 
\usepackage{mathtools}

\usepackage{subcaption}

\usepackage{array} 
\usepackage{newtxtext,newtxmath} 
\usepackage[cal=boondoxo,scr=euler]{mathalfa}
\usepackage[backref=page,linktocpage]{hyperref} 
\usepackage{cleveref} 
\usepackage{caption} %
\usepackage{graphics,graphicx} 
\usepackage{tikz,tikz-cd} 

\usetikzlibrary{graphs,graphs.standard,calc}

\usetikzlibrary{decorations.pathreplacing,}

\usepackage{enumerate} 

\DeclareMathAlphabet{\mathsf}{OT1}{\sfdefault}{m}{n}

\newcommand{\nocontentsline}[3]{}
\newcommand{\tocless}[2]{\bgroup\let\addcontentsline=\nocontentsline#1{#2}\egroup}

\makeatletter
\@ifpackageloaded{newtxmath}{%
\renewcommand{\lambda}{\uplambda}%
\renewcommand{\zeta}{\upzeta}
\renewcommand{\xi}{\upxi}
\renewcommand{\tau}{\uptau}
\renewcommand{\alpha}{\upalpha}
}{
\usepackage{mlmodern}
}
\makeatother

\usepackage[margin=1.30in]{geometry}
\usepackage{verbatim}

\usepackage{scalerel}

\makeatletter
\def\dual#1{\expandafter\dual@aux#1\@nil}
\def\dual@aux#1/#2\@nil{\begin{tabular}{@{}c@{}}#1\\#2\end{tabular}}
\makeatother

\makeatletter
\@namedef{subjclassname@2020}{\textup{2020} Mathematics Subject Classification}
\makeatother

\DeclareMathAlphabet{\amathbb}{U}{bbold}{m}{n}

\hypersetup{
    colorlinks = true,
    linkbordercolor = {white},
    linkcolor = {BrickRed},
    anchorcolor = {black},
    citecolor = {BrickRed},
    filecolor = {cyan},
    menucolor = {BrickRed},
    runcolor = {cyan},
    urlcolor = {black}
}

\usetikzlibrary{automata}
\usepackage{bbm}

\newtheoremstyle{teoremas}
{11pt}
{11pt}
{\itshape}
{}
{\bfseries}
{}
{.5em}
{}

\theoremstyle{teoremas}
\newtheorem{theorem}{Theorem}[section]

\newtheorem{lemma}[theorem]{Lemma}
\newtheorem{proposition}[theorem]{Proposition}

\newtheoremstyle{definition}
{11pt}
{11pt}
{}
{}
{\bfseries}
{}
{.5em}
{}

\theoremstyle{definition}

\newtheorem{conjecture}[theorem]{Conjecture}

\newtheorem{example}[theorem]{Example}
\newtheorem{remark}[theorem]{Remark}

\crefname{theorem}{theorem}{theorems}
\Crefname{theorem}{Theorem}{Theorems}
\crefname{lemma}{lemma}{lemmas}
\Crefname{lemma}{Lemma}{Lemmas}
\crefname{proposition}{proposition}{propositions}
\Crefname{proposition}{Proposition}{Propositions}

\DeclareMathOperator{\rk}{rk}

\newcommand{\M}{\mathsf{M}}
\newcommand{\rev}{\operatorname{rev}}

\newcommand{\U}{\mathsf{U}}

\newcommand{\Z}{\mathbb{Z}}
\newcommand{\LL}{\mathsf{\Lambda}}

\newcommand{\C}{\mathsf{C}}

\newcommand{\rank}{\operatorname{rk}}

\newcommand{\cL}{\mathcal{L}}
\newcommand{\cH}{\mathcal{H}}
\newcommand{\PG}{\operatorname{PG}}

\newcommand{\qbinom}[2]{\begin{bmatrix} #1 \\ #2 \end{bmatrix}_q}

\AtBeginDocument{%
   \def\MR#1{}
}

\title[Inverse Kazhdan--Lusztig polynomials of matroids under deletion]{Inverse Kazhdan--Lusztig polynomials\\ of matroids under deletion}

\author[T.~Braden]
{Tom~Braden}

\address{(T.~Braden)
University of Massachusetts Amherst, Amherst, MA, USA
}
\email{braden@umass.edu}

\author[L.~Ferroni]{Luis~Ferroni}

\address{(L.~Ferroni)
  Dipartimento di Matematica, Universit\`a di Pisa, Pisa, Italy.
}
\email{luis.ferroni@unipi.it}

\thanks{During the early stages of this project, Luis Ferroni was a member at the Institute for Advanced Study, supported by the Minerva Research Foundation. Jacob Matherne and Nutan Nepal were supported by NSF Grant DMS-2452179 and Simons Foundation Travel Support for Mathematicians Award MPS-TSM00007970.}

\author[J.~P.~Matherne]{Jacob~P.~Matherne}

\address{(J. P. Matherne)
Department of Mathematics, North Carolina State University, Raleigh, NC, USA.
}
\email{jpmather@ncsu.edu}

\author[N.~Nepal]{Nutan Nepal}

\address{(N. Nepal)
Department of Mathematics, North Carolina State University, Raleigh, NC, USA.
}
\email{nnepal2@ncsu.edu}


\allowdisplaybreaks

\begin{document}

\begin{abstract}
We provide a deletion formula for the inverse Kazhdan--Lusztig polynomial and the inverse $Z$-polynomial of a matroid. Our formulas provide analogues to the deletion formulas of Braden--Vysogorets for Kazhdan--Lusztig and $Z$-polynomials.  We discuss several consequences, which include closed formulas and recursions for these invariants on uniform matroids, projective geometries, glued cycles, and arbitrary matroids of corank $2$. As a relevant application of our deletion formula, we show the existence of a matroid of rank $19$ which disproves a conjecture of Xie and Zhang concerning a real-rootedness property of inverse Kazhdan--Lusztig polynomials.
\end{abstract}

\maketitle

\section{Introduction}

The Kazhdan--Lusztig polynomial $P_{\M}(x)$ is a fundamental invariant associated with any matroid $\M$. It was introduced by Elias, Proudfoot, and Wakefield in \cite{elias-proudfoot-wakefield}, and exhibits formal similarities to the Kazhdan--Lusztig polynomials defined for Bruhat intervals of Coxeter groups. The coefficients of $P_{\M}(x)$ depend only on the lattice of flats $\mathcal{L}(\M)$ of the matroid and, in fact, they can be expressed as signed integer combinations of the flag Whitney numbers counting chains of flats with specified ranks (see \cite{proudfoot-xu-young}).

The Kazhdan--Lusztig polynomial of $\M$ plays a significant role in the \emph{singular} Hodge theory of matroids developed by Braden, Huh, Matherne, Proudfoot, and Wang in \cite{braden-huh-matherne-proudfoot-wang}. Two further polynomial invariants that are often studied in this context are the \emph{inverse Kazhdan--Lusztig polynomial of a matroid}, introduced by Gao and Xie \cite{gao-xie} and denoted $Q_{\M}(x)$, and the $Z$-polynomial introduced by Proudfoot, Xu, and Young \cite{proudfoot-xu-young} and denoted $Z_{\M}(x)$. One of the main contributions in \cite{braden-huh-matherne-proudfoot-wang} consists of interpreting these three polynomials in the following way:
\begin{itemize}
    \item $P_{\M}(x)$ is the Hilbert series of the stalk of the intersection cohomology module of $\M$ at the empty set \cite[Theorem 1.9]{braden-huh-matherne-proudfoot-wang}.
    \item $Z_{\M}(x)$ is the Hilbert series of the intersection cohomology module of $\M$ \cite[Theorem 1.9]{braden-huh-matherne-proudfoot-wang}.
    \item The coefficient of $x^k$ in $Q_{\M}(x)$ is the multiplicity of the trivial module in the degree $\rk(\M) - 2k$ piece of the Rouquier complex of $\M$ \cite[Proposition 8.21]{braden-huh-matherne-proudfoot-wang}.
\end{itemize}

The above interpretations imply that these three polynomials have non-negative coefficients for every matroid $\M$. However, computing the above polynomials is often a daunting task even for small matroids.

Very recently, Gao, Ruan, and Xie \cite{gao-ruan-xie} studied a fourth polynomial, called the \emph{inverse $Z$-polynomial}, denoted by $Y_{\M}(x)$. While the nonnegativity of the coefficients of $Y_{\M}(x)$ follows from the nonnegativity for $Q_{\M}(x)$, at the moment we lack an interpretation for the coefficients of $Y_{\M}(x)$ as graded dimensions (see \cite[Problem~1]{gao-ruan-xie}).

In \cite{braden-vysogorets}, Braden and Vysogorets presented a formula that relates the Kazhdan--Lusztig polynomial
of a matroid $\M$ to that of the matroid $\M\smallsetminus i$ obtained by deleting an element $i$, as well as various restrictions and contractions of $\M$. Specifically, for a matroid $\M$ where $i$ is not a coloop, their main result is the following.

\begin{theorem}[{\cite[Theorem~2.8]{braden-vysogorets}}]\label{thm:bv}
    Let $\M$ be a loopless matroid of rank $k$, and let $i\in E$ be an element of the ground set that is not a coloop. Then,
    \begin{align*}
    P_{\M}(x) &= P_{\M\smallsetminus\{i\}}(x) - xP_{\M/\{i\}}(x) + \sum_{F\in\mathscr{S}_i} \tau\left(\M/{(F\cup\{i\})}\right)\, x^{\frac{k - \rk(F)}{2}}\, P_{\M|_F}(x),\\
     Z_{\M}(x) &= Z_{\M\smallsetminus\{i\}}(x) + \sum_{F\in\mathscr{S}_i} \tau\left(\M/{(F\cup\{i\})}\right)\, x^{\frac{k - \rk(F)}{2}}\, Z_{\M|_F}(x).
    \end{align*}
\end{theorem}

\noindent In the above statement, the notation ${\mathscr{S}}_i$ stands for a distinguished set of flats:
\begin{equation}\label{eq:S_i}
    {\mathscr{S}}_i = {\mathscr{S}}_i(\M) = \left\{ F \in \mathcal{L}(\M): F\subsetneq E\smallsetminus\{i\} \text{ and } F\cup\{i\}\in \mathcal{L}(\M)\right\},
\end{equation}
while 
\[ \tau(\M) := 
    \begin{cases}
        [x^{\frac{\rk(\M)-1}{2}}] P_{\M}(x) & \text{if $\rk(\M)$ is odd,}\\
        0 & \text{if $\rk(\M)$ is even.}
    \end{cases}
\]

\noindent In addition to the intrinsic interest in computing these polynomials, the community has shown considerable excitement about when they might exhibit log-concavity and real-rootedness properties. 

\begin{conjecture}\label{conj:four-conjectures}
    For every matroid $\M$, the following properties hold:
    \begin{enumerate}[(a)]
        \item \label{it:conj-p} \cite[Conjecture~3.2]{gedeon-proudfoot-young-survey} The polynomial $P_{\M}(x)$ is real-rooted.
        \item \label{it:conj-z}\cite[Conjecture~5.1]{proudfoot-xu-young} The polynomial $Z_{\M}(x)$ is real-rooted.
        \item \label{it:conj-q}\cite[Conjecture~4.2]{gao-xie} The polynomial $Q_{\M}(x)$ has log-concave coefficients.
        \item \label{it:conj-y}\cite[Conjecture~1.5]{gao-ruan-xie} The polynomial $Y_{\M}(x)$ has log-concave coefficients.
    \end{enumerate}
\end{conjecture}

For \eqref{it:conj-p}, \eqref{it:conj-q}, and \eqref{it:conj-y}, the existing evidence consists mostly of explicit computations for special families of matroids, e.g., wheels and whirls \cite{wheels-whirls}, uniform matroids \cite{xie-zhang-uniform,gao-li-xie-yang-zhang}, thagomizer matroids \cite{gedeon-thagomizer,wu-zhang}, and sparse paving matroids \cite{ferroni-vecchi,ferroni-nasr-vecchi,gao-ruan-xie}.

For \eqref{it:conj-z} there is a general partial result, proved by Ferroni, Matherne, Stevens, and Vecchi in \cite[Theorem~4.7]{ferroni-matherne-stevens-vecchi}: this asserts that $Z$-polynomials of matroids are $\gamma$-positive, a result that was conjectured in \cite{ferroni-nasr-vecchi}. This is weaker than real-rootedness, but stronger than unimodality\footnote{The unimodality of $Z$-polynomials follows from the hard Lefschetz theorem on the intersection cohomology module, proved by Braden, Huh, Matherne, Proudfoot, and Wang, see \cite[Theorem~1.2]{braden-huh-matherne-proudfoot-wang}.}. We emphasize that a key role in that proof is played by the deletion formula for $Z$-polynomials appearing in Theorem~\ref{thm:bv}.

A refinement of Conjecture~\ref{conj:four-conjectures}\eqref{it:conj-q} was recently proposed by Xie and Zhang \cite{xie-zhang-conjecture}. In order to state their conjecture we introduce some terminology. If $Q_{\M}(x) = q_0 + q_1x + \cdots + q_s x^s$ where $s = \deg Q_{\M} \leq \lfloor \frac{\rk(\M)-1}{2}\rfloor$, then we define the \emph{normalization} of $Q_{\M}(x)$ as the polynomial
    \[ \mathcal{B}(Q_{\M})(x) = \binom{s}{0}q_0 + \binom{s}{1}q_1\,x + \cdots + \binom{s}{s} q_s\,x^s.\]
That is, the polynomial $\mathcal{B}(Q_\M)(x)$ is the Hadamard product of $Q_{\M}(x)$ and $(1+x)^s$. 

\begin{conjecture}[{\cite[Conjecture~1.4]{xie-zhang-conjecture}}]\label{conj:xie-zhang}
    For every matroid $\M$, the polynomial $\mathcal{B}(Q_{\M})(x)$ is real-rooted.
\end{conjecture}

The last conjecture would imply, via the Newton inequalities on real-rooted polynomials, the ultra log-concavity of the coefficients of $\mathcal{B}(Q_{\M})(x)$. This, in turn, is equivalent to the log-concavity of the coefficients of $Q_{\M}(x)$ predicted by Conjecture~\ref{conj:four-conjectures}\eqref{it:conj-q}. To support their own conjecture, Xie and Zhang proved that it holds for all paving matroids, a class of matroids that is conjecturally predominant  \cite[Conjecture~1.6]{mayhew-etal}. 

\subsection{Main results}

Motivated by Theorem~\ref{thm:bv}, we were led to consider a deletion formula for both the inverse Kazhdan--Lusztig polynomials $Q_{\M}(x)$ and the inverse $Z$-polynomials $Y_{\M}(x)$. Even though each of these polynomials, up to a sign, arise by inverting, respectively, $P_{\M}(x)$ and $Z_{\M}(x)$ in the incidence algebra of the lattice of flats $\mathcal{L}(\M)$ (see Section~\ref{sec:recap-kl} for more details), it is not at all clear how to deduce the desired deletion formulas from Theorem~\ref{thm:bv}. 

In order to state the deletion formulas for $Q_{\M}(x)$ and $Y_{\M}(x)$ we need to introduce a counterpart for the sets $\mathscr{S}_i$ appearing in equation~\eqref{eq:S_i}. We define
    \begin{equation}\label{eq:T_i}
        \mathscr{T}_i = \mathscr{T}_i(\M) := \left\{ F \in \mathcal{L}(\M): i\in F \text{ and } F\smallsetminus\{i\}\notin \mathcal{L}(\M)\right\}.
    \end{equation}

\begin{theorem}\label{thm:deletion_formula_Q}
    Let $\M$ be a loopless matroid on $E$, and let $i\in E$ be an element that is not a coloop. Then, the following identities hold:
    \begin{align*} 
    Q_{\M}(x) &= Q_{\M\smallsetminus i}(x) + (1+x) Q_{\M/i}(x) - \sum_{F\in \mathcal{T}_i} \tau(\M|_F/i)\, x^{\rk(F)/2}\, Q_{\M/F}(x),\\
    Y_{\M}(x) &= Y_{\M\smallsetminus i}(x) + (1+x) Y_{\M/i}(x) - \sum_{F\in \mathcal{T}_i} \tau(\M|_F/i)\, x^{\rk(F)/2}\, Y_{\M/F}(x).
    \end{align*}
\end{theorem}

Our formulas lead to new recursive computations of $Q_{\M}(x)$ and $Y_{\M}(x)$ for any matroid $\M$, starting from the base case of Boolean matroids $\U_{n,n}$. For these matroids, we have $Q_{\U_{n,n}}(x) = 1$ and $Y_{\U_{n,n}}(x) = (x+1)^n$. 
Note that the deletion formulas for $Q_{\M}(x)$ and $Y_{\M}(x)$ are exactly the same, a phenomenon that does not occur in Theorem~\ref{thm:bv}.

While there is a geometric intuition behind Theorem~\ref{thm:bv} (see the discussion in \cite[Section~1.1]{braden-vysogorets}), we lack that intuition in the present framework. In a sense, the proof of Theorem~\ref{thm:deletion_formula_Q} is purely algebraic, though combinatorics joins the game when we apply properties of incidence algebras of posets in our computations.

\medskip

As a notable application of Theorem~\ref{thm:deletion_formula_Q}, we are able to show (see Theorem~\ref{thm:counterexample-body} below) the existence of a matroid of rank $19$ which disproves Conjecture~\ref{conj:xie-zhang}.

\begin{theorem}\label{thm:counterexample-main}
    There exists a matroid $\M$ of rank $19$ on $21$ elements whose normalized inverse Kazhdan--Lusztig polynomial $\mathcal{B}(Q_{\M})(x)$ is not real-rooted.
\end{theorem}

Let us explain how we were led to the construction of the counterexample. Three crucial ingredients were needed:
    \begin{itemize}
        \item The valuativity of the inverse Kazhdan--Lusztig polynomial.
        \item A closed formula for the inverse Kazhdan--Lusztig polynomial of the graphic matroid associated to the graph consisting of two cycles glued along an edge.
        \item A formula for the inverse Kazhdan--Lusztig polynomial of an arbitrary matroid of corank $2$.
    \end{itemize}

The first item in the above list is a result proved by Ardila and Sanchez \cite[Theorem~8.8]{ardila-sanchez}. The second item is a consequence of Theorem~\ref{thm:deletion_formula_Q} that we derive in the present paper (see Proposition~\ref{prop:glued-cycles}). The third, stated as Theorem~\ref{thm:crank2}, relies on the machinery of Ferroni and Schr\"oter \cite{ferroni-schroter} on valuative invariants of elementary split matroids.

Besides these two main results, we also provide in Section~\ref{sec:examples} and Section~\ref{sec:counterexample} several computational examples. They include formulas for our polynomials in the cases of uniform matroids, glued cycles, projective geometries with an element deleted, and corank $2$ matroids.

\section{Preliminaries}

We will assume familiarity with the essentials of matroid theory. For undefined terminology on matroids and geometric lattices, we refer to \cite{welsh,oxley}. Unless indicated otherwise, in this paper we will typically assume that our matroid $\M$ is loopless, its ground set is $E$, and the rank of a subset $A\subseteq E$ is denoted by $\rk_{\M}(A)$, or simply $\rk(A)$ if $\M$ is understood by context.

We will make consistent use of the following notions. Throughout this paper, we denote by $\mathcal{L}(\M)$ the lattice of flats of any matroid.  The \emph{characteristic polynomial} of a loopless matroid $\M$ is the polynomial $\chi_{\M}(x) \in \mathbb{Z}[x]$ given by
    \[\chi_{\M}(x) := \sum_{F\in\mathcal{L}(\M)} \mu(\varnothing, F)\, x^{\rk(\M)-\rk(F)},\]
where the number $\mu(\varnothing,F)$ is the value assigned by the M\"obius function of $\mathcal{L}(\M)$ to the closed interval $[0,F]\subseteq \mathcal{L}(\M)$ (see \cite[Chapter~3]{stanley-ec1}). A relevant specialization of the characteristic polynomial is the so-called \emph{M\"obius invariant} of $\M$, which is in turn defined by
    \[ \mu(\M) := \mu(\varnothing,E) = \chi_{\M}(0).\]

\subsection{Incidence algebras and kernels} The \emph{incidence algebra} of $\mathcal{L}(\M)$ over a commutative unitary ring $R$, denoted by $\mathcal{I}_R$, is the free $R$-module spanned by all the closed intervals of $\mathcal{L}(\M)$. In other words, an element $a\in \mathcal{I}_R$ associates to each pair of flats $F\subseteq G$ of $\M$ an element $a_{FG}\in R$. The product (also known as convolution) of two elements $a,b\in \mathcal{I}_R$ is defined via

    \[ (ab)_{FG} = \sum_{\substack{H\in\mathcal{L}(\M)\\F\subseteq H\subseteq G}} a_{FH}\, b_{HG}, \qquad \text{ for every $F\subseteq G$ in $\mathcal{L}(\M)$}.\]

The algebra $\mathcal{I}_R$ satisfies the following basic properties:

\begin{enumerate}[(i)]
    \item The product in $\mathcal{I}_R$ is associative.
    \item There is a multiplicative identity $\delta\in \mathcal{I}_R$ defined by
        \[ \delta_{FG} = \begin{cases} 1 & \text{ if $F = G$},\\ 0 & \text{ if $F \neq G$}.\end{cases}\]
    \item An element $a\in \mathcal{I}_R$ is invertible if and only if $a_{FF}$ is invertible in $R$ for every $F\in \mathcal{L}(\M)$.
\end{enumerate}

Let us consider the case in which $R = \mathbb{Z}[x,x^{-1}]$, and let us denote $\mathcal{I} = \mathcal{I}_R$. There is an involution $a\mapsto a^{\rev}$ on $\mathcal{I}$ defined by
    \[ (a^{\rev})_{FG}(x) = x^{\rk_{\M}(G) - \rk_{\M}(F)}\, a_{FG}(x^{-1}),\]
which ``reverses'' coefficients.

An element $\kappa\in \mathcal{I}$ is said to be a \emph{kernel} whenever $\kappa^{\rev} = \kappa^{-1}$. It is a well-known fact that the element $\chi\in \mathcal{I}$ assigning to each closed interval $[F,G]$ the characteristic polynomial of the matroid $\M|_G/F$ is a kernel. (See \cite{stanley-local,proudfoot-kls,ferroni-matherne-vecchi} for a thorough discussion of the notion of kernels and for further examples.)

\subsection{A recapitulation on the Kazhdan--Lusztig invariants}\label{sec:recap-kl}

The primary objects of study in the present paper are the four polynomials mentioned in the introduction. For the sake of completeness, we include below two theorems which play the role of definitions for them.

\begin{theorem}[{\cite{proudfoot-xu-young,braden-vysogorets}}]\label{thm:definition-kl-and-zeta}
    There is a unique way to assign to each loopless matroid $\M$ polynomials $P_{\M}(x), Z_{\M}(x) \in \Z[x]$ such that the following properties hold:
    \begin{enumerate}[\normalfont(i)]
        \item If $\rk(\M) = 0$, then $P_{\M}(x) = 1$.
        \item If $\rk(\M) > 0$, then $\deg P_{\M}(x) < \frac{1}{2} \rk(\M)$.
        \item For every matroid $\M$, the polynomial
            \[ Z_{\M}(x) := \sum_{F\in \mathcal{L}(\M)} x^{\rk(F)}\, P_{\M/F}(x)\]
        is palindromic.
    \end{enumerate}
\end{theorem}

The polynomial $P_{\M}(x)$ is called the \emph{Kazhdan--Lusztig polynomial} of $\M$, and was first introduced by Elias, Proudfoot, and Wakefield in \cite{elias-proudfoot-wakefield}. The polynomial $Z_{\M}(x)$ is called the \emph{$Z$-polynomial} of $\M$ and was considered first in the work of Proudfoot, Xu, and Young \cite{proudfoot-xu-young}.

\begin{theorem}[{\cite{braden-huh-matherne-proudfoot-wang,ferroni-matherne-stevens-vecchi, gao-ruan-xie}}]\label{thm:definition-q-and-y}
    There is a unique way to assign to each loopless matroid $\M$ polynomials $Q_{\M}(x), Y_{\M}(x) \in \Z[x]$ such that the following properties hold:
    \begin{enumerate}[\normalfont(i)]
        \item If $\rk(\M) = 0$, then $Q_{\M}(x) = 1$.
        \item If $\rk(\M) > 0$, then $\deg Q_{\M}(x) < \frac{1}{2} \rk(\M)$.
        \item For every matroid $\M$, the polynomial
            \[ Y_{\M}(x) := \sum_{F\in \mathcal{L}(\M)}(-1)^{\rk(\M) - \rk(F)}\mu(\M/F) \, x^{\rk(\M) - \rk(F)}\, Q_{\M|_F}(x)\]
        is palindromic.
    \end{enumerate}
\end{theorem}

The polynomial $Q_{\M}(x)$ is called the \emph{inverse Kazhdan--Lusztig polynomial} of $\M$, and was first considered by Gao and Xie in \cite{gao-xie}. The polynomial $Y_{\M}(x)$ is called the \emph{inverse $Z$-polynomial} of $\M$. It was first mentioned in the work of Ferroni, Matherne, Stevens, and Vecchi \cite{ferroni-matherne-stevens-vecchi}, and it was later addressed in more detail in the recent work by Gao, Ruan, and Xie \cite{gao-ruan-xie}. 

The reason these polynomials carry the word ``inverse'' in their name is because they can be defined from the Kazhdan--Lusztig and $Z$-polynomials by working in the incidence algebra of $\mathcal{L}(\M)$ over $R=\mathbb{Z}[x,x^{-1}]$. Defining $P\in \mathcal{I}$ as the element associating to each interval $[F,G]\subseteq \mathcal{L}(\M)$ the polynomial
    \[ P_{FG}(x) = P_{\M|_G/F}(x),\]
and defining $Q,Z,Y\in \mathcal{I}$ analogously,  the following relationships hold:
    \[ P^{-1} = \widehat{Q} \qquad \text{ and } \qquad Z^{-1} = \widehat{Y}, \]
where $\widehat{Q}_{\M}(x) = (-1)^{\rk(\M)} Q_{\M}(x)$ and $\widehat{Y}_{\M}(x) = (-1)^{\rk(\M)}\, Y_{\M}(x)$. Put more succinctly, the element $Q$ (resp $Y$) is, up to a sign, the inverse of $P$ (resp. $Z$) in the incidence algebra of $\mathcal{L}(\M)$.

\section{Proof of the deletion formulas}

\subsection{The Kazhdan--Lusztig basis} Let $\M$ be a matroid and $\cL(\M)$ be its lattice of flats. Following \cite{braden-vysogorets}, we define the free $\mathbb{Z}[x, x^{-1}]$-module $\cH(\M)$  generated by the elements $\Gamma_{\M}^F$, one for each flat $F\in \cL(\M)$.
Elements of $\cH(\M)$ are formal sums of the form
\begin{equation}\label{eq:alpha}
    \alpha = \sum_{F \in \cL(\M)} \alpha_F(x) \cdot \Gamma_{\M}^F, \quad \alpha_F \in \mathbb{Z}
    [x, x^{-1}].
\end{equation}
In general, an element in the module $\cH(\M)$ will be identified by the presence of a subindex\footnote{In \cite{braden-vysogorets}, the element $\Gamma_{\M}^F$ is denoted by $F$. Our change of notation, carrying $\M$ as a subindex, bypasses any potential confusion when the same set $F$ is a flat of two different matroids.} $\M$. The basis $\left\{\Gamma^F_{\M}\right\}_{F\in\mathcal{L}(\M)}$ will often be referred to as the standard basis of $\cH(\M)$, but several other bases will be used throughout our proofs.

There is an involution in $\mathcal{H}(\M)$, denoted $\alpha \mapsto \overline{\alpha}$, and defined for any element $\alpha$ as in equation~\eqref{eq:alpha} by
    \[ \overline{\alpha} := \sum_{F\in\mathcal{L}(\M)} \alpha_F(x^{-1}) \, \overline{\Gamma^F_{\M}},\]
where $\overline{\Gamma^F_{\M}}$ is defined by
    \[ \overline{\Gamma^F_{\M}} := \sum_{\substack{G\in \mathcal{L}(\M)\\ G\subseteq F}} x^{2(\rk(G) - \rk(F))} \, \chi_{\M|_F/G}(x^2)\cdot \Gamma^G_{\M}.\]
The assertion $\overline{\overline{\Gamma^F_{\M}}} =\Gamma^F_{\M}$ (and therefore that $\alpha$ is indeed an involution) is formally equivalent to the fact that the characteristic polynomial is a kernel in the incidence algebra of $\mathcal{L}(\M)$. Notice that, in each summand of the above display the first two factors correspond to $\chi^{\rev}_{FG}(x^{-2})$.

There is a special $\mathbb{Z}$-submodule $\mathcal{H}_p(\M)\subseteq \mathcal{H}(\M)$ which consists of all the elements fixed by the above involution.  The following was proved by Braden and Vysogorets in \cite[Proposition~2.13]{braden-vysogorets}. 

\begin{proposition}
    There exists a $\mathbb{Z}$-basis (called the Kazhdan--Lusztig basis) of $\mathcal{H}_p(\M)$ given by the collection $\left\{ \zeta_{\M}^F : F\in\mathcal{L}(\M)\right\}$, where
\begin{equation}
    \label{eq:zeta_definition}
    \zeta_{\M}^F := \sum_{\substack{G \in \mathcal{L}(\M)\\ G \subseteq F}} x^{\rk(F) - \rk(G)} P_{\M|_F/G}(x^{-2}) \cdot \Gamma^G_{\M}.
\end{equation}
\end{proposition}

It is a straightforward fact that the collection $\{\zeta^F_{\M}\}_{F\in\mathcal{L}(\M)}$ forms a $\mathbb{Z}[x,x^{-1}]$-basis of the free module $\cH(\M)$. The fact that the elements are in $\mathcal{H}_p(\M)$ is equivalent to property (iii) of Theorem \ref{thm:definition-kl-and-zeta}. The next lemma expresses the standard basis $\left\{\Gamma^F_{\M}\right\}_{F\in\mathcal{L}(\M)}$ of $\cH(\M)$ in terms of the Kazhdan--Lusztig basis. In order to state this change of basis explicitly, we make use of the following notation:

\begin{equation}
    \widehat{Q}_\M(x) := (-1)^{\rk(\M)} Q_\M(x).
\end{equation}

\begin{lemma}\label{lem:standard_basis_in_zeta}
    For every flat $F\in\mathcal{L}(\M)$, the following equality holds in $\mathcal{H}(\M)$:
    \begin{equation}
        \label{cob-lemma}
        \Gamma^F_{\M} = \sum_{\substack{G \in \mathcal{L}(\M)\\ G \subseteq F}} x^{\rk(F) - \rk(G)} \widehat{Q}_{\M|_F/G}(x^{-2}) \cdot \zeta^G_\M.
    \end{equation}
\end{lemma}

\begin{proof}
    Since $P$ and $\widehat{Q}$ are inverses in $\mathcal{I}$, we have
    \[
        \sum_{\substack{G, H \in \mathcal{L}(\M)\\ H \subseteq G \subseteq F}}{P_{\M|_F/G}(x)\cdot \widehat{Q}_{\M|_G/H}(x)} =
        \sum_{\substack{G, H \in \mathcal{L}(\M)\\ H \subseteq G \subseteq F}}{\widehat{Q}_{\M|_F/G}(x)\cdot P_{\M|_G/H}(x)} =
        \delta_{FH},
    \]
    which is $1$ when $F=H$ and $0$ otherwise. Using this fact, we note that the change-of-basis matrices defined by equations \eqref{eq:zeta_definition} and
    \eqref{cob-lemma} are inverses of each other.
\end{proof}

\subsection{A homomorphism for single element deletions}\label{sec:delta}

We start by recalling a few elementary facts observed in \cite[Section~2.6]{braden-vysogorets}. 
Whenever $i$ is an element of the ground set of $\M$ that is not a coloop, there is a surjective map $\cL(\M) \to \cL(\M\smallsetminus i)$ sending a flat $F \in \cL(\M)$ to $F\smallsetminus i \in \cL(\M\smallsetminus i)$. This map induces a $\mathbb{Z}[x,x^{-1}]$-linear map $\Delta\colon \cH(\M) \to \cH(\M\smallsetminus i)$, which is defined on the basis elements $\Gamma^F_{\M} \in \cH(\M)$ by \[\Delta(\Gamma^F_{\M}):= x^{\rk_{\M\smallsetminus i}(F\smallsetminus i)- \rk_\M(F)} \cdot \Gamma^{F\smallsetminus i}_{\M\smallsetminus i}.\] Written more explicitly:

\begin{equation}
    \Delta(\Gamma^F_{\M})  = \begin{cases}
        \Gamma^F_{\M\smallsetminus i}                          & \text{if }F\not\ni i,                                      \\
        \Gamma^{F\smallsetminus i}_{\M\smallsetminus i}               & \text{if } F\ni i \text{ and }F\smallsetminus i \notin \cL(\M), \\
        x^{-1}\cdot \Gamma^{F\smallsetminus i}_{\M\smallsetminus i}  & \text{if } F\ni i \text{ and }F\smallsetminus i \in \cL(\M).
    \end{cases}
\end{equation}

We now state and prove a crucial lemma, which is central to the proof of Theorem~\ref{thm:deletion_formula_Q}. It describes the behavior of $\Delta$ on the Kazhdan--Lusztig basis.

\begin{lemma}
    \label{lem:delta_zeta_nocoloop}
    Let $\M$ be a matroid on $E$, and let $i\in E$ be an element that is not a coloop. 
    \begin{enumerate}[\normalfont (i)]
        \item If $F \in \cL(\M)$ is a flat such that $i \in F$ and $F\smallsetminus i \notin \cL(\M)$, then
    \begin{equation}
        \Delta(\zeta_{\M}^F) = \zeta^{F\smallsetminus i}_{\M\smallsetminus i} + \sum_{G\in \mathscr{S}_i(\M|_F)}\tau(\M|_F/(G\cup i))\cdot\zeta^G_{\M\smallsetminus i},
    \end{equation}
    where $\mathscr{S}_i(\M|_F) = \{G\in \cL(\M|_F) : i\notin G \text{ and } G\cup i \in \cL(\M|_F)\}$.
    \item \label{lem:delta_zeta_special_case} If $F \in \cL(\M)$ is a flat such that $i \in F$ and $F\smallsetminus i \in \cL(\M)$, then
    \[ \Delta(\zeta^F_{\M}) = (x + x^{-1})\zeta_{\M\smallsetminus i}^{F\smallsetminus i} .\]
    \end{enumerate}
\end{lemma}

\begin{proof}
    The first of the two cases was already proved by Braden and Vysogorets and appears as \cite[Equation~(12)]{braden-vysogorets}. 
    
    We focus on the second case. Let us denote $F_0 = F\smallsetminus i$. We are assuming that $F_0$ is a flat of $\M$.
    Using the definition of $\Delta$, we analyze the possible values of $\Delta(\Gamma_{\M}^G)$ for $G\subseteq F$ a flat of $\M$:
    \begin{itemize}
        \item If $G \subseteq F$ and $i \notin G$, then $G \subseteq F_0 \in \cL(\M\smallsetminus i)$. Since $G \in \cL(\M)$ and $i \notin G$, it follows that $G$ is also a flat of $\M\smallsetminus i$. In this case, $\Delta(\Gamma_{\M}^G) = \Gamma_{\M\smallsetminus i}^G$.
        \item If $G \subseteq F$ and $i \in G$, let $G_0 = G\smallsetminus i$. Then $G_0 \subseteq F_0 \in \cL(\M\smallsetminus i)$. Since $F_0 \in \cL(\M)$, it follows that $G_0 = G \cap F_0 \in \cL(\M)$. In this case, $\Delta(\Gamma_{\M}^G) = x^{-1}\Gamma_{\M\smallsetminus i}^{G_0}$.
    \end{itemize}
    We apply these in the following computation of $\Delta(\zeta_{\M}^F)$:
    \begin{align}
    \Delta(\zeta_{\M}^F) &= \Delta \left( \sum_{\substack{G \in \mathcal{L}(\M)\\ G \subseteq F}} x^{\rk(F) - \rk(G)} P_{\M|_F/G}(x^{-2}) \cdot \Gamma_{\M}^G\right) \nonumber \\
    &= \sum_{\substack{G \in \mathcal{L}(\M)\\ G \subseteq F\\i\notin G}}  x^{\rk(F) - \rk(G)} P_{\M|_F/G}(x^{-2}) \cdot \Delta(\Gamma_{\M}^G) +\sum_{\substack{G \in \mathcal{L}(\M)\\ G \subseteq F\\i\in G}}  x^{\rk(F) - \rk(G)} P_{\M|_F/G}(x^{-2}) \cdot \Delta(\Gamma_{\M}^G)  \nonumber\\
    &= \sum_{\substack{G\in \mathcal{L}(\M)\\G \subseteq F_0}} x^{\rk(F) - \rk(G)} P_{\M|_F/G}(x^{-2}) \cdot \Gamma_{\M\smallsetminus i}^G \nonumber \\
    &\qquad\qquad + \sum_{\substack{G_0\in \mathcal{L}(\M)\\G_0 \subseteq F_0}} x^{\rk(F) - \rk(G_0 \cup \{i\})} P_{\M|_F/(G_0 \cup \{i\})}(x^{-2}) \cdot (x^{-1}\Gamma_{\M\smallsetminus i}^{G_0}).\label{eq:two-summands}
    \end{align}
    In the last equation, every appearance of the rank function is as the rank function of $\M$. We can relate this to the rank function of $\M\smallsetminus i$ as follows:
    \begin{itemize}
        \item $\rk_{\M}(F) = \rk_{\M\smallsetminus i}(F_0) + 1$.
        \item For $G \subseteq F_0$, we have $\rk_{\M}(G) = \rk_{\M\smallsetminus i}(G)$.
        \item For $G_0 \subseteq F_0$, we have $\rk_{\M}(G_0 \cup \{i\}) = \rk_{\M\smallsetminus i}(G_0) + 1$.
    \end{itemize}
    We also have Kazhdan--Lusztig polynomial identities under these conditions:
    \begin{itemize}
        \item In the case in which $G \subseteq F_0$, we have an isomorphism $\M|_F/G \cong (\M\smallsetminus i)|_{F_0}/G \oplus \U_{1,1}$, and therefore we obtain $P_{\M|_F/G}(x^{-2}) = P_{(\M\smallsetminus i)|_{F_0}/G}(x^{-2})$.
        \item In the case in which $G_0 \subseteq F_0$, we have an isomorphism $\M|_F/(G_0 \cup \{i\}) \cong (\M\smallsetminus i)|_{F_0}/{G_0}$, and therefore $P_{\M|_F/(G_0 \cup \{i\})}(x^{-2}) = P_{(\M\smallsetminus i)|_{F_0}/G_0}(x^{-2})$.
    \end{itemize}
    Substituting these identities into the two summands in equation~\eqref{eq:two-summands}, we have
    \begin{align*}
    \Delta(\zeta_{\M}^F) &= \sum_{\substack{G\in \mathcal{L}(\M\smallsetminus i)\\G \subseteq F_0}} x^{(\rk_{\M\smallsetminus i}(F_0) + 1) - \rk_{\M\smallsetminus i}(G)} P_{(\M\smallsetminus i)|_{F_0}/G}(x^{-2}) \cdot \Gamma_{\M\smallsetminus i}^G \\
    &\qquad\qquad + x^{-1} \sum_{\substack{G_0\in \mathcal{L}(\M\smallsetminus i)\\G_0 \subseteq F_0}} x^{(\rk_{\M\smallsetminus i}(F_0) + 1) - (\rk_{\M\smallsetminus i}(G_0) + 1)} P_{(\M\smallsetminus i)|_{F_0}/G_0}(x^{-2}) \cdot \Gamma_{\M\smallsetminus i}^{G_0} \\
    &= x \sum_{\substack{G\in \mathcal{L}(\M\smallsetminus i)\\G \subseteq F_0}} x^{\rk_{\M\smallsetminus i}(F_0) - \rk_{\M\smallsetminus i}(G)} P_{(\M\smallsetminus i)|_{F_0}/G}(x^{-2}) \cdot \Gamma_{\M\smallsetminus i}^G \\
    &\qquad\qquad + x^{-1} \sum_{\substack{G_0\in \mathcal{L}(\M\smallsetminus i)\\G_0 \subseteq F_0}} x^{\rk_{\M\smallsetminus i}(F_0) - \rk_{\M\smallsetminus i}(G_0)} P_{(\M\smallsetminus i)|_{F_0}/G_0}(x^{-2}) \cdot \Gamma_{\M\smallsetminus i}^{G_0}.
    \end{align*}
    Note that the two sums appearing in the last equation are equal. Furthermore, they agree with the definition of $\zeta_{\M\smallsetminus i}^{F_0} = \zeta_{\M\smallsetminus i}^{F\smallsetminus i}$. This completes the proof.
\end{proof}

\subsection{The deletion formula for inverse KL polynomials}

We have all the necessary ingredients to prove the first part of Theorem~\ref{thm:deletion_formula_Q}. 

\newtheorem*{thm:intro1}{Theorem~\ref{thm:deletion_formula_Q} (first part)}
\begin{thm:intro1}
    Let $\M$ be a loopless matroid on $E$, and let $i\in E$ be an element that is not a coloop. The following identity holds:
    \[ Q_{\M}(x) = Q_{\M\smallsetminus i}(x) + (1+x) Q_{\M/i}(x) - \sum_{F\in \mathcal{T}_i} \tau(\M|_F/i)\, x^{\rk(F)/2}\, Q_{\M/F}(x).\]
\end{thm:intro1}

\begin{proof}
    We use Lemma~\ref{lem:standard_basis_in_zeta} for $F = E$ the full ground set, and obtain
    \[ \Gamma_{\M}^E = \sum_{G\in \mathcal{L}(\M)} x^{\rk(E) - \rk(G)} \widehat{Q}_{\M/G}(x^{-2})\, \zeta^G_{\M}.\]
    Applying the homomorphism $\Delta$ yields
    \begin{equation}\label{eq:two-sides}
    \Gamma_{\M\smallsetminus i}^{E\smallsetminus i} = \sum_{G\in\mathcal{L}(\M)} x^{\rk(E) - \rk(G)} \widehat{Q}_{\M/G}(x^{-2}) \cdot \Delta(\zeta_{\M}^G).
    \end{equation}
    By writing both sides of equation \eqref{eq:two-sides} in the Kazhdan--Lusztig basis, we may equate the resulting coefficients of $\zeta_{\M\smallsetminus i}^{\varnothing}$. On the left-hand side we apply Lemma~\ref{lem:standard_basis_in_zeta} again, but now for the matroid $\M\smallsetminus i$ and the flat $E\smallsetminus i$. We obtain that the coefficient of $\zeta_{\M\smallsetminus i}^\varnothing$ on the left-hand side equals
    $x^{\rk_{\M\smallsetminus i}(E\smallsetminus i) - \rk_{\M\smallsetminus i}(\varnothing)}\cdot
        \widehat{Q}_{\M|_{E\smallsetminus i}}(x^{-2}) = x^{\rk(\M)}\cdot
        \widehat{Q}_{\M\smallsetminus i}(x^{-2})$.

    To compute the coefficient of $\zeta^{\varnothing}_{\M\smallsetminus i}$ on the right-hand side of equation~\eqref{eq:two-sides}, we employ Lemma~\ref{lem:delta_zeta_nocoloop}. Examining the summand for $G\in \mathscr{L}(\M)$, this coefficient is non-zero only when $G\in \{\varnothing, i\} \cup \mathscr{T}_i$. In each of these cases, we have the following:
    \begin{itemize}
        \item For $G = \varnothing$, we have $\Delta(\zeta_{\M}^\varnothing) = \zeta_{\M\smallsetminus i}^\varnothing$.
        \item For $G = i$, we have $\Delta(\zeta_{\M}^i) = (x+x^{-1})\zeta_{\M\smallsetminus i}^\varnothing$.
        \item For $G\in \mathscr{T}_i$, we have $\Delta(\zeta_{\M}^G) = \tau(\M|_G/i)\cdot\zeta_{\M\smallsetminus i}^\varnothing
                  +(\text{terms not including $\zeta_{\M\smallsetminus i}^\varnothing$})$.
    \end{itemize}

    Collecting all these terms, the coefficient of $\zeta_{\M\smallsetminus i}^{\varnothing}$ on the right-hand side of equation~\eqref{eq:two-sides} equals
    \begin{equation}
         x^{\rk(\M)}
        \widehat{Q}_{\M}(x^{-2}) + x^{\rk(\M)- 1}(x+x^{-1})
        \widehat{Q}_{\M/i}(x^{-2})          +
         \sum_{G\in \mathscr{T}_i} x^{\rk(\M) - \rk(G)}
        \widehat{Q}_{\M/G}(x^{-2})\cdot \tau(\M|_G/i).
    \end{equation}
    Now, our result follows by equating this with the computation we made of this coefficient on the left-hand side of equation \eqref{eq:two-sides}:
    \begin{align*}x^{\rk(\M)}\cdot
        \widehat{Q}_{\M\smallsetminus i}(x^{-2}) &= x^{\rk(\M)}
        \widehat{Q}_{\M}(x^{-2}) + x^{\rk(\M)- 1}(x+x^{-1})
        \widehat{Q}_{\M/i}(x^{-2})+\\
        &\qquad\qquad
         \sum_{G\in \mathscr{T}_i} x^{\rk(\M) - \rk(G)}
        \widehat{Q}_{\M/G}(x^{-2})\cdot \tau(\M|_G/i).
    \end{align*}
    By rearranging terms, applying the substitution $\widehat{Q}_{\M}(x) = (-1)^{\rk(\M)} Q_{\M}(x)$, and using the change of variables $x^{-2} = t$, we are led to the formula we desired.
\end{proof}

\subsection{Deletion formula for the inverse \texorpdfstring{$Z$}{Z}-polynomial}
In order to prove the deletion formula for $Y_{\M}(x)$, we consider the related polynomial $\widehat{Y}_{\M}(x)$ given by
\[
    \widehat{Y}_{\M}(x) := (-1)^{\rk(\M)}\, Y_{\M}(x) = \sum_{F \in \mathcal{L}(\M)} \left( x^{\rk(\M/F)}\mu(\M/F) \right) \cdot \widehat{Q}_{\M|_F}(x).
\]
We define a new family of elements in $\mathcal{H}(\M)$:
\[\xi_{\M}^F := \sum_{\substack{G\in\mathcal{L}(\M)\\G\subseteq F}} x^{\rk(G)}\mu(\M|_F/G)\,\cdot \Gamma_{\M}^G.\] 

Because the M\"obius invariant of the rank zero matroid is $1$, these elements form a new $\Z[x,x^{-1}]$-basis of 
$\cH(\M)$. The next lemma shows how the deletion map $\Delta$ interacts with these elements.

\begin{lemma}\label{lem:deletion-xi}
   Let $\M$ be a matroid on $E$, and let $i\in E$ be an element of $\M$.
   \begin{enumerate}[\normalfont (i)]
   \item \label{xi-with-i-coloop}
   If $i$ is a coloop of $\M$, then \(\Delta(\xi_{\M}^E) = 0\).
   \item \label{xi-with-i-noncoloop}
   If $i$ is not a coloop of $\M$, then \(\Delta(\xi_{\M}^E) = \xi_{\M\smallsetminus i}^{E\smallsetminus i}\).
   \end{enumerate}
\end{lemma}

\begin{proof}
    Since $\Delta$ is $\mathbb{Z}[x,x^{-1}]$-linear, we have
    \[
        \Delta(\xi_{\M}^E) = \sum_{G\in \mathcal{L}(\M)} x^{\rk(G)}\mu(\M/G)\cdot \Delta(\Gamma_{\M}^G).
    \]
    We partition the lattice of flats $\cL(\M)$ into four disjoint sets: 
    \begin{itemize}
        \item $A = \{G\in \cL(\M): i \notin G \text{ and } G\cup i \notin \cL(\M)\}$.
        \item $S = \{G\in \cL(\M): i \notin G \text{ and } G\cup i \in \cL(\M)\}$.
        \item $T = \{G\in \cL(\M): i \in G \text{ and } G\smallsetminus i \in \cL(\M)\}$.
        \item $B = \{G\in \cL(\M): i \in G \text{ and } G\smallsetminus i \notin \cL(\M)\}$.
    \end{itemize}
    We can now rewrite the sum by splitting it over these four sets:
    \begin{multline*}
        \Delta(\xi^E_{\M})  = \sum_{G \in S} x^{\rk(G)}\mu(\M/G)\cdot \Delta(\Gamma_{\M}^G) + \sum_{G \in T} x^{\rk(G)}\mu(\M/G)\cdot \Delta(\Gamma_{\M}^G)       \\
                       \qquad+ \sum_{G \in A} x^{\rk(G)}\mu(\M/G)\cdot \Delta(\Gamma_{\M}^G) + \sum_{G \in B} x^{\rk(G)}\mu(\M/G)\cdot \Delta(\Gamma_{\M}^G).
    \end{multline*}
    Applying the definition of $\Delta$ to each sum, one obtains:
    \begin{itemize}
        \item For $G \in S \cup A$ (flats not containing $i$), we have $\Delta(\Gamma_{\M}^G) = \Gamma_{\M\smallsetminus i}^{G\smallsetminus i}$.
        \item For $G \in T$, we have $\Delta(\Gamma_{\M}^G) = x^{-1}\Gamma_{\M\smallsetminus i}^{G\smallsetminus i}$.
        \item For $G \in B$, we have $\Delta(\Gamma_{\M}^G) = \Gamma_{\M\smallsetminus i}^{G\smallsetminus i}$.
    \end{itemize}
    Substituting these into the expression above gives
    \begin{multline*}
        \Delta(\xi^E_{\M})  = \sum_{G \in S} x^{\rk(G)}\mu(\M/G)\cdot \Gamma^G_{\M\smallsetminus i} + \sum_{G \in T} x^{\rk(G)-1}\mu(\M/G)\cdot \Gamma_{\M\smallsetminus i}^{G\smallsetminus i}     \\
                      + \sum_{G \in A} x^{\rk(G)}\mu(\M/G)\cdot \Gamma^G_{\M\smallsetminus i} + \sum_{G \in B} x^{\rk(G)}\mu(\M/G)\cdot \Gamma_{\M\smallsetminus i}^{G\smallsetminus i}.
    \end{multline*}
    There is a natural bijection between $S$ and $T$ where for each $F \in S$, we have $F \cup i \in T$. We re-index the sum over $T$ using $F \in S$, which lets us combine the first two sums:
    \begin{multline*}
        \Delta(\xi^E_{\M}) = \sum_{F \in S} \left( x^{\rk(F)}\mu(\M/F) + x^{\rk(F \cup i)-1}\mu(\M/(F\cup i)) \right) \cdot \Gamma_{\M\smallsetminus i}^F            \\
                      + \sum_{G \in A} x^{\rk(G)}\mu(\M/G)\cdot \Gamma^G_{\M\smallsetminus i} + \sum_{G \in B} x^{\rk(G)}\mu(\M/G)\cdot \Gamma_{\M\smallsetminus i}^{G\smallsetminus i}.
    \end{multline*}
    For $F\in S$, note that $\rk(F \cup i) = \rk(F) + 1$.

    We note that in the case where $i$ is a coloop of $\M$, the sets $A$ and $B$ are empty. We also have that $i$ is a coloop in the contraction $\M/F$ for $F\in S$. Thus $\mu(\M/F) = -\mu(\M/(F\cup i))$ and \(\Delta(\xi_{\M}^E)\) simplifies to $0$, proving the statement \eqref{xi-with-i-coloop}.

    In the case where $i$ is not a coloop of $\M$, since $i$ is neither a loop nor a coloop of the contraction $\M/F$, we have the standard deletion-contraction relation for the M\"obius function:
    \[
        \mu(\M/F) + \mu(\M/(F\cup i)) = \mu((\M\smallsetminus i)/F).
    \]
    There are analogous relations for the flats in the other sets. For a flat $G \in A$, the interval $[G,E]$ in $\mathcal{L}(\M)$ is isomorphic to the interval $[G, E\smallsetminus i]$ in $\mathcal{L}(\M\smallsetminus i)$, which gives
    \[
        \mu(\M/G) = \mu((\M\smallsetminus i)/G).
    \]
    For a flat $G \in B$, the set $F = G\smallsetminus i$ is a flat in $\M\smallsetminus i$ but not in $\M$. 
    The intervals $[G, E]$ and $[G\smallsetminus i, E\smallsetminus i]$ are isomorphic, so the relationship between the Möbius functions in this case is
    \[
        \mu(\M/G) = \mu((\M\smallsetminus i)/(G\smallsetminus i)).
    \]
    Substituting these identities into the expression for $\Delta(\xi^E_\M)$ yields
    \begin{multline*}
        \Delta(\xi_{\M}^E)  = \sum_{F \in S} x^{\rk(F)}\mu((\M\smallsetminus i)/F)\cdot \Gamma_{\M\smallsetminus i}^F                                                                                                     
                      + \sum_{G \in A} x^{\rk(G)}\mu((\M\smallsetminus i)/G)\cdot \Gamma_{\M\smallsetminus i}^G \\ + \sum_{G \in B} x^{\rk(G)}\mu((\M\smallsetminus i)/(G\smallsetminus i))\cdot \Gamma^{G\smallsetminus i}_{\M\smallsetminus i}.
    \end{multline*}
    Since the terms on the right account exactly for all flats of the matroid $\M\smallsetminus i$, we obtain
    \[
        \Delta(\xi_{\M}^E) = \xi_{\M\smallsetminus i}^{E\smallsetminus i}. \qedhere
    \]
\end{proof}

We are now ready to prove the second part of Theorem~\ref{thm:deletion_formula_Q}.

\newtheorem*{thm:intro2}{Theorem~\ref{thm:deletion_formula_Q}~(second part)}
\begin{thm:intro2}
    Let $\M$ be a loopless matroid on $E$, and let $i\in E$ be an element that is not a coloop. The following identity holds:
    \[ Y_{\M}(x) = Y_{\M\smallsetminus i}(x) + (1+x) Y_{\M/i}(x) - \sum_{F\in \mathcal{T}_i} \tau(\M|_F/i)\, x^{\rk(F)/2}\, Y_{\M/F}(x).\]
\end{thm:intro2}

\begin{proof}
By applying Lemma~\ref{lem:standard_basis_in_zeta} to the definition of the element $\xi^E_{\M}$, we can write it in terms of the Kazhdan--Lusztig basis elements:
\begin{align}
    \xi_{\M}^E & = \sum_{F \in \mathcal{L}(\M)} x^{\rk(F)}\mu(\M/F)\left( \sum_{\substack{G \in \mathcal{L}(\M)\\ G \subseteq F}} x^{\rk(F) - \rk(G)} \widehat{Q}_{\M|_F/G}(x^{-2}) \cdot \zeta^G_\M \right)         \nonumber\\
          & = \sum_{F \in \mathcal{L}(\M)} \sum_{\substack{G\in \mathcal{L}(\M)\\G\subseteq F}} x^{2\rk(F)-\rk(G)}\mu(\M/F)\cdot\widehat{Q}_{\M|_F/G}(x^{-2}) \cdot \zeta_{\M}^G                            \nonumber\\
          & = \sum_{G \in \mathcal{L}(\M)} \sum_{\substack{F\in \mathcal{L}(\M)\\F\supseteq G}} x^{2\rk(F)-\rk(G)}\mu(\M/F)\cdot\widehat{Q}_{\M|_F/G}(x^{-2}) \cdot \zeta_{\M}^G                            \nonumber\\
          & = \sum_{G \in \mathcal{L}(\M)} x^{2\rk(E)-\rk(G)}\left(\sum_{\substack{F \in \mathcal{L}(\M)\\ F\supseteq G} }x^{2\rk(F)-2\rk(E)}\mu(\M/F)\widehat{Q}_{\M|_F/G}(x^{-2})\right) \cdot \zeta_{\M}^G \nonumber\\
          & = \sum_{G \in \mathcal{L}(\M)} x^{2\rk(E)-\rk(G)}\widehat{Y}_{\M/G}(x^{-2}) \cdot \zeta_{\M}^G. \label{eq:xi-kl-basis}
\end{align}
Applying the deletion map $\Delta$ to this expression and using Lemma~\ref{lem:deletion-xi}, we get
\begin{equation}\label{eq:xi-deletion-2}
    \xi_{\M\smallsetminus i}^{E\smallsetminus i} = \Delta(\xi_{\M}^E)                      = \sum_{G \in \mathcal{L}(\M)} x^{2\rk(E)-\rk(G)}\widehat{Y}_{\M/G}(x^{-2}) \cdot \Delta(\zeta_{\M}^G).
\end{equation}
By applying equation~\eqref{eq:xi-kl-basis} to the matroid $\M\smallsetminus i$ with ground set $E\smallsetminus i$, we obtain that the coefficient of $\zeta_{\M\smallsetminus i}^\varnothing$ on the left-hand side of equation~\eqref{eq:xi-deletion-2} equals $x^{2\rk(E\smallsetminus i)}\widehat{Y}_{\M\smallsetminus i}(x^{-2})$.
Like in the proof of Theorem~\ref{thm:deletion_formula_Q} (first part), the coefficient of $\zeta_{\M\smallsetminus i}^\varnothing$ on the right-hand side of equation~\eqref{eq:xi-deletion-2} is
\[
    x^{2\rk(E)}\widehat{Y}_{\M}(x^{-2}) - (x+x^{-1})x^{2\rk(E)-1}\widehat{Y}_{\M/i}(x^{-2}) +
    \sum_{G \in \mathscr{T}_i} \tau(\M|_G/i)\cdot x^{2\rk(E)-\rk(G)}\widehat{Y}_{\M/G}(x^{-2}).
\]
Equating the coefficients of $\zeta_{\M\smallsetminus i}^\varnothing$ on both sides, we get
\[
    \widehat{Y}_{\M\smallsetminus i}(x^{-2}) = \widehat{Y}_{\M}(x^{-2}) - (1+x^{-2})\widehat{Y}_{\M/i}(x^{-2}) +
    \sum_{G \in \mathscr{T}_i} \tau(\M|_G/i)\cdot x^{-\rk(G)}\widehat{Y}_{\M/G}(x^{-2}).
\]
After applying the change of variables $t = x^{-2}$, rearranging terms, and using that $\widehat{Y}_{\M}(x) = (-1)^{\rk(\M)} Y_{\M}(x)$, the desired deletion formula follows.
\end{proof}

\section{Some special cases and examples}\label{sec:examples}

In order to illustrate its use, in this section we apply Theorem~\ref{thm:deletion_formula_Q} to a few special cases.

\subsection{Uniform matroids} 

The following is a new recursion that allows for the computation of inverse Kazhdan--Lusztig and inverse $Z$-polynomials of uniform matroids.

\begin{proposition}\label{prop:deletion_formula_Q_uniform}
Let $\M = \U_{k,n}$ be a uniform matroid of rank $k$ on $n$ elements, where $0 < k < n$. Then,
\begin{align*}
    Q_{\U_{k,n}}(x) &= Q_{\U_{k,n-1}}(x) + (1 + x)\cdot Q_{\U_{k-1,n-1}}(x) -
    \tau(\U_{k-1,n-1})\cdot x^{k/2},\\
    Y_{\U_{k,n}}(x) &= Y_{\U_{k,n-1}}(x) + (1 + x)\cdot Y_{\U_{k-1,n-1}}(x) -
    \tau(\U_{k-1,n-1})\cdot x^{k/2}.
\end{align*}
\end{proposition}
\begin{proof}
    Since $k < n$, none of the elements in the ground set are coloops; thus, we may apply the deletion formula for a generic element $i$ in the ground set. The set $\mathscr{T}_i$ then consists only of the top flat $E$, and the sum over $\mathscr{T}_i$ reduces exactly to the sum in our statement.
\end{proof}

The reader may compare the above identities with the following formulas for $Q_{\U_{k,n}}(x)$ from \cite[Theorem~3.3]{gao-xie}, and for $Y_{\U_{k,n}}(x)$ from \cite[Theorem~1.1]{gao-ruan-xie}:

\begin{align*}
    Q_{\U_{k,n}}(x) &= \binom{n}{k} \sum_{j=0}^{\lfloor \frac{k-1}{2}\rfloor} \frac{(n-k)(k-2j)}{(n-k+j)(n-j)}\binom{k}{j}\, x^j,\\
    Y_{\U_{k,n}}(x) &= \sum_{i=0}^{\lfloor\frac{k}{2}\rfloor} \binom{n}{i}\binom{n-i-1}{n-k}\, x^i + \sum_{i=0}^{\lfloor\frac{k-1}{2}\rfloor} \binom{n}{i}\binom{n-i-1}{n-k}\, x^{k-i}.
\end{align*}

The first of the above two formulas, along with a result due to Vecchi in \cite[Theorem~4.1]{vecchi}, implies that
    \[ \tau(\U_{k,n}) = \begin{cases} \displaystyle\binom{n}{k}\binom{k}{\lfloor \frac{k-1}{2}\rfloor} \dfrac{4(n-k)}{(2n-k-1)(2n-k+1)}& \text{if $k$ is odd,}\\
    0 & \text{if $k$ is even.}\end{cases}\]

\subsection{Glued cycles}

Let us denote by $\C_{n}$ the graphic matroid associated to a cycle of length $n$, i.e., $\C_n = \U_{n-1,n}$. Let $\mathsf{C}_{a,b}$ be the graphic matroid that results from gluing two cycles of length $a$ and $b$ along a common edge (see Figure~\ref{fig:two-cycles} for an example). Deleting this common edge results in the graphic matroid $\C_{a+b-1}$. In particular, our deletion formula reads as follows.

\begin{proposition}\label{prop:glued-cycles}
    The inverse Kazhdan--Lusztig and inverse $Z$-polynomials of the graphic matroids $\C_{a,b}$ are given by
    \begin{align*} 
    Q_{\mathsf{C}_{a,b}}(x) &= Q_{\U_{n-2,n-1}}(x)  + (1+x) Q_{\U_{a-2,a-1}}(x)\, Q_{\U_{b-2,b-1}}(x) \\
    &\qquad\qquad- \left(\tau(\U_{a-2,a-1})\, x^{\frac{a-1}{2}}\, Q_{\U_{b-2,b-1}}(x) + \tau(\U_{b-2,b-1})\, x^{\frac{b-1}{2}} \,Q_{\U_{a-2,a-1}}(x)\right),\\
    Y_{\mathsf{C}_{a,b}}(x) &= Y_{\U_{n-2,n-1}}(x)  + (1+x) Y_{\U_{a-2,a-1}}(x)\, Y_{\U_{b-2,b-1}}(x) \\
    &\qquad\qquad- \left(\tau(\U_{a-2,a-1})\, x^{\frac{a-1}{2}}\, Y_{\U_{b-2,b-1}}(x) + \tau(\U_{b-2,b-1})\, x^{\frac{b-1}{2}} \,Y_{\U_{a-2,a-1}}(x)\right),
    \end{align*}
    where $n = a+b-1$.
\end{proposition}

\begin{proof}
 Since the deletion formulas for $Q$ and $Y$ look the same, we indicate the proof only for $Q$, as the other follows analogously. Let us call $i$ the edge along which the cycles of size $a$ and $b$ are glued to form $\C_{a,b}$. Note that the set $\mathscr{T}_i$ in $\C_{a,b}$ consists of the full ground set $E$, the two ground sets of $\C_a = \U_{a-1,a}$ and $\C_b = \U_{b-1,b}$, and several other flats for which $\tau(\M|_F/i)$ vanishes due to the presence of coloops.
    
    Since $\C_{a,b}/i$ is isomorphic to the disconnected matroid $\U_{a-2,a-1} \oplus \U_{b-2, b-1}$, we have $\tau(\C_{a,b}/i) = 0$ by  \cite[Lemma 2.7]{braden-vysogorets}.
    As noted in the discussion prior to the statement, $\C_{a,b}\smallsetminus i = \C_{a+b-2}$, and thus $Q_{\C_{a,b}\smallsetminus i}(x) = Q_{\U_{a+b-3, a+b-2}}(x)$. By putting the pieces together and using that $Q$ and $Y$ are multiplicative under direct sums (see \cite[Lemma~3.1]{gao-xie} and \cite[Proposition~2.3]{gao-ruan-xie}), we obtain the formula of the statement.
\end{proof}

    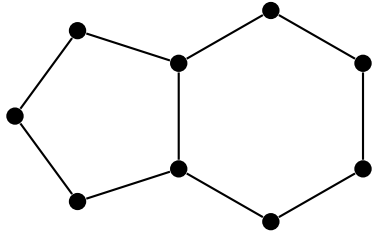
\begin{figure}[ht]
    \centering
    \begin{tikzpicture}
        [scale=0.7,auto=center,
         every node/.style={circle, fill=black, inner sep=2.3pt},
         every edge/.append style = {thick}]
        \tikzstyle{edges} = [thick];

        \node (a1) at (0,0) {};
        \node (a2) at (0,2) {};
        \node (a3) at (1.7320,3) {};
        \node (a4) at (2*1.7320,2) {};
        \node (a5) at (2*1.7320,0) {};
        \node (a6) at (2*1.7320/2,-1) {};
        \draw[edges] (a1) -- (a2);
        \draw[edges] (a2) -- (a3);
        \draw[edges] (a3) -- (a4);
        \draw[edges] (a4) -- (a5);
        \draw[edges] (a5) -- (a6);
        \draw[edges] (a6) -- (a1);

        \node (p1) at (-1.902113,  2.618034) {};
        \node (p2) at (-3.077684,  1.000000) {};
        \node (p3) at (-1.902113, -0.618034) {};
        \draw[edges] (a2) -- (p1);
        \draw[edges] (p1) -- (p2);
        \draw[edges] (p2) -- (p3);
        \draw[edges] (p3) -- (a1);

    \end{tikzpicture}
    \caption{The graphic matroid $\mathsf{C}_{5,6}$.}
    \label{fig:two-cycles}
\end{figure}

\begin{remark}
    When the numbers $a$ and $b$ are even, the $\tau$ terms appearing in the preceding proposition vanish. Therefore, the two formulas in the above proposition become simpler: 
    \begin{align*} 
    Q_{\mathsf{C}_{a,b}}(x) &= Q_{\U_{n-2,n-1}}(x)  + (1+x) Q_{\U_{a-2,a-1}}(x)\, Q_{\U_{b-2,b-1}}(x), \\
    Y_{\mathsf{C}_{a,b}}(x) &= Y_{\U_{n-2,n-1}}(x)  + (1+x) Y_{\U_{a-2,a-1}}(x)\, Y_{\U_{b-2,b-1}}(x).
    \end{align*}
\end{remark}

\subsection{Single element deletions of projective geometries}

In this section we employ Theorem~\ref{thm:definition-q-and-y} in order to compute the inverse Kazhdan--Lusztig polynomial of a projective geometry with a single element deleted.

Let $V=V(r, q)$
    be the r-dimensional vector space over the finite field $\mathbb{F}_q$.
We denote the projective geometry $P(V)$ of rank $r$ over the finite field $\mathbb{F}_q$ by $\M = \PG(r-1,q)$.  This is a modular matroid, and so its Kazhdan--Lusztig polynomial is the constant $1$ (see \cite[Proposition~2.14]{elias-proudfoot-wakefield}). An unpublished result of Vecchi \cite[Theorem~3.5]{vecchi} guarantees that the inverse Kazhdan--Lusztig polynomial of $\PG(r-1,q)$ is a constant polynomial, and in fact equals $\mu(\PG(r-1,q)) = q^{\binom{r}{2}}$.

\begin{proposition}
    Consider  the projective geometry $\M = \PG(r-1, q)$ of rank $r\geq 2$ over the finite field $\mathbb{F}_q$. Then, 
    \[ Q_{\M\smallsetminus i}(x)  =q^{\binom{r}{2}} - q^{\binom{r-1}{2}} + \left(\frac{q^{r-1}-1}{q-1} q^{\binom{r-2}{2}} - q^{\binom{r-1}{2}}\right) \cdot x.\]
    
\end{proposition}
\begin{proof}
    We start from the formula for $Q_{\M}(x)$ in Theorem \ref{thm:deletion_formula_Q} and substitute $\M = \PG(r-1, q)$.
    The lattice of flats of a projective geometry is modular.

    The sum in Theorem \ref{thm:deletion_formula_Q} is over the set
    $\mathscr{T}_i = \{G \in \cL(\M) : i \in G \text{ and } G\smallsetminus i \notin \mathcal{L}(\M)\}$.
    For the projective geometry $\M$, a flat $G$ containing $i$ is a linear subspace. If $\rank(G) \geq 2$, removing the point $i$ results in a set $G\smallsetminus i$ that is not a subspace, and thus not a flat. If $\rank(G)=1$, then $G=\{i\}$ and $G\smallsetminus i = \varnothing$, which is a flat. Thus, for $\M=\PG(r-1,q)$, the set $\mathscr{T}_i$ consists of all flats containing $i$ of rank $2$ or greater.

    For each $G \in \mathscr{T}_i$, the term in the sum is $\tau(\M|_G/i)\cdot x^{\rk(\M|_G/i)/2} \cdot Q_{\M/G}(x)$.
    The matroid $\M|_G/i$ has lattice of flats isomorphic to the interval $[i, G]$ in $\cL(\M)$, and thus isomorphic to $\PG(k-2,q)$ where $k=\rk(G)$.
    The rank of $\M|_G/i$ is $k-1$ and  so $\tau(\M|_G/i)$ is non-zero only if the rank $k-1=1$, which implies $k=2$.

    Thus, the sum is over the set of rank 2 flats containing the point $i$.
    The number of such flats is the number of 2-dimensional subspaces of $V(r,q)$ containing a given 1-dimensional subspace, which is equal to the number of $1$-dimensional subspaces in the quotient space $V(r,q)/\langle i \rangle \cong V(r-1,q)$. This number is $\qbinom{r-1}{1}=\frac{q^{r-1}-1}{q-1}$, the number of lines through the point $i$ in $\M$. For each such rank 2 flat $G$:
    \begin{itemize}
        \item $\tau(\M|_G/i) = \tau(\PG(0,q)) = 1$, since $\PG(0,q)$ has rank 1.
        \item $Q_{\M/G}(x) = Q_{\PG(r-3,q)}(x)$, since the contraction $\M/G$ of $\M=\PG(r-1,q)$ by a rank~2 flat $G$ is isomorphic to $\PG(r-3,q)$.
    \end{itemize}
    The entire sum therefore reduces to
    \[
        \sum_{G \in \mathscr{T}_i, \rk(G)=2} \tau(\M|_G/i)\cdot x^{\rk(G)/2} \cdot Q_{\M/G}(x) = \qbinom{r-1}{1} \cdot x \cdot Q_{\PG(r-3,q)}(x).
    \]
    Substituting this back into the general formula from Theorem \ref{thm:deletion_formula_Q}, and using the fact that the contraction $\M/i$ is isomorphic to $\PG(r-2,q)$, we get
    \[
        Q_{\PG(r-1,q)}(x) = Q_{\M\smallsetminus i}(x) + (1 + x)\cdot Q_{\PG(r-2,q)}(x) - \qbinom{r-1}{1} \cdot x \cdot Q_{\PG(r-3,q)}(x).
    \]
    Substituting $Q_{\PG(r-1,q)}(x) = q^{\binom{r}{2}}$ (and correspondingly for $Q_{\PG(r-2,q)}(x)$ and $Q_{\PG(r-3,q)}(x)$), we are led to the desired identity after rearranging some terms.
\end{proof}

\section{A counterexample to Conjecture~\ref{conj:xie-zhang}}\label{sec:counterexample}

Let $\M$ be a matroid of rank $2$ without loops. The lattice of flats of $\M$ has the empty set as its bottom element and $E$ as its top element. This guarantees that the rank $1$ flats of $\M$ form a partition of $E$. Conversely, any partition of $E$ yields a loopless rank $2$ matroid on $E$. 

By taking duals, the above construction yields a procedure that builds a corank $2$ matroid on $E$ without coloops from every set partition of $E$.

\begin{example}
    Consider $n=8$ and the partition of the set $\{1,\ldots,8\}$ consisting of the parts $\{1,5,6,8\}$, $\{2\}$, and $\{3,4,7\}$. The loopless rank $2$ matroid that has these three sets as its rank $1$ flats is the graphic matroid induced by the graph depicted in Figure~\ref{fig:graph}.  (If there are more than three parts in the partition, these matroids fail to be graphic.) The corank $2$ matroid associated to this set partition is the corresponding matroid dual.
    \begin{figure}[ht] 
		\begin{tikzpicture} 
		[scale=3.0,auto=center,every node/.style={circle, fill=black, inner sep=2.3pt}] 
		\tikzstyle{edges} = [black, thick];
		
		\node (a1) at (-0.5,0.86) {};  
		\node (a3) at (0,1.73)  {}; 
		\node (a2) at (0.5,0.86)  {};    
		
		\draw[edges] (a1) -- (a2) ;
		\draw[edges] (a3) edge[bend right=20] (a1) ;
		\draw[edges] (a3) edge[bend right=-20] (a1);
		\draw[edges] (a3) edge[bend right=8] (a1) ;
		\draw[edges] (a3) edge[bend right=-8] (a1);
		\draw[edges] (a3) edge[bend right=15] (a2);
		\draw[edges] (a3) edge[bend right=-15] (a2);
		\draw[edges] (a3) -- (a2);
		\end{tikzpicture}
		\caption{The edges on the left are labelled by $\{1,5,6,8\}$, the edges on the right by $\{3,4,7\}$ and the bottom one by $\{2\}$.}\label{fig:graph}
	\end{figure}
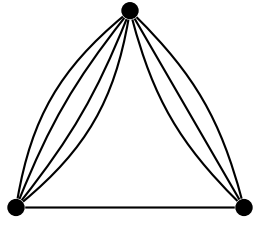
\end{example}

Notice that rank $2$ matroids without loops are trivially paving, and therefore, corank $2$ matroids without coloops are copaving.  This fact is relevant because the class of elementary split matroids considered by Ferroni and Schr\"oter in \cite{ferroni-schroter} contains all copaving matroids.

Furthermore, since Ardila and Sanchez \cite[Theorem~8.4]{ardila-sanchez} showed that the inverse Kazhdan--Lusztig polynomial is a valuative invariant, we may apply \cite[Theorem~1.4]{ferroni-schroter} to compute $Q_{\M}(x)$ and $Y_{\M}(x)$ for any corank $2$ matroid without coloops. In order to state our next result, let us recall that in \cite{ferroni-schroter} a subset $A\subseteq E$ is said to be \emph{stressed} if both $\M/A$ and $\M|_A$ are uniform matroids.

As the following result shows (once combined with Proposition~\ref{prop:glued-cycles}), our formulas ultimately only depend on the values of $Q_{\M}(x)$ and $Y_{\M}(x)$ for uniform matroids of corank $1$ and $2$.

\begin{theorem}\label{thm:crank2}
     Let $\M$ be a corank $2$ matroid without coloops. For each $r$, let $\lambda_r$ denote the number of stressed subsets of rank $r$ and size $r+1$ in $\M$. Then
    \begin{align*}
        Q_{\M}(x) & = Q_{\U_{n-2,n}}(x) - \sum_{r} \lambda_{r} \sum_{a=2}^{n-r-1} \left(Q_{\C_{a,n+1-a}}(x) - Q_{\U_{a-1,a}}(x)\, Q_{\U_{n-a-1,n-a}}(x)\right),\\
        Y_{\M}(x) &= Y_{\U_{n-2,n}}(x) - \sum_{r} \lambda_{r} \sum_{a=2}^{n-r-1} \left(Y_{\C_{a,n+1-a}}(x) - Y_{\U_{a-1,a}}(x)\, Y_{\U_{n-a-1,n-a}}(x)\right).
    \end{align*}
    In particular, if $\M$ arises from a partition $\lambda$ of $[n]$, then
    \begin{align*}
        Q_{\M}(x) & = Q_{\U_{n-2,n}}(x) - \sum_{\substack{S\in \lambda\\ |S|\geq 2}} \sum_{a=2}^{|S|} \left(Q_{\C_{a,n+1-a}}(x) - Q_{\U_{a-1,a}}(x)\, Q_{\U_{n-a-1,n-a}}(x)\right),\\
        Y_{\M}(x) &= Y_{\U_{n-2,n}}(x) - \sum_{\substack{S\in \lambda\\ |S|\geq 2}}\sum_{a=2}^{|S|} \left(Y_{\C_{a,n+1-a}}(x) - Y_{\U_{a-1,a}}(x)\, Y_{\U_{n-a-1,n-a}}(x)\right).
    \end{align*}
\end{theorem}

\begin{proof}
    By a direct application of \cite[Theorem~5.3]{ferroni-schroter}, we obtain that for any coloopless corank $2$ matroid on $n$ elements, the following equality holds:
    \[ Q_{\M}(x) = Q_{\U_{n-2,n}}(x) - \sum_{r,h} \lambda_{r,h} \left( Q_{\LL_{r,n-2,h,n}}(x) - Q_{\U_{n-2-r,n-h}}(x) Q_{\U_{r,h}}(x)\right),\]
    where $\lambda_{r,h}$ is the number of stressed subsets of $\M$ that have rank $r$ and size $h$. (We omit the definition of the matroid $\LL_{r,n-2,h,n}$ because it is inessential for this proof, as we will see in what follows.)
    
    Since $\M$ has corank $2$, the definition of being a stressed subset implies that $h\in \{r,r+1,r+2\}$, and in fact the cases $h=r$ and $h=r+2$ cannot happen. Thus,
    \[ Q_{\M}(x) = Q_{\U_{n-2,n}}(x) - \sum_{r} \lambda_{r} \left( Q_{\LL_{r,n-2,r+1,n}}(x) - Q_{\U_{n-r-2,n-r-1}}(x) Q_{\U_{r,r+1}}(x)\right),\]
    where now $\lambda_r$ counts the number of stressed subsets of rank $r$ and size $r+1$.   By \cite[Corollary~5.7]{ferroni-schroter}, since the inverse Kazhdan--Lusztig polynomial is a valuative invariant, we obtain
    \[ Q_{\LL_{r,n-2,r+1,n}}(x) = \sum_{a=2}^{n-r-1} Q_{\C_{a,n+1-a}}(x) - \sum_{a=2}^{n-r-2} Q_{\U_{a-1,a}}(x)\, Q_{\U_{n-a-1,n-a}}(x),\]
    which yields
    \[ Q_{\LL_{r,n-2,r+1,n}}(x)- Q_{\U_{n-r-2,n-r-1}}(x) Q_{\U_{r,r+1}}(x) = \sum_{a=2}^{n-r-1} \left(Q_{\C_{a,n+1-a}}(x) - Q_{\U_{a-1,a}}(x)\, Q_{\U_{n-a-1,n-a}}(x)\right).\]
    Putting all the pieces together, we obtain
    \[ Q_{\M}(x) = Q_{\U_{n-2,n}}(x) - \sum_{r} \lambda_{r} \sum_{a=2}^{n-r-1} \left(Q_{\C_{a,n+1-a}}(x) - Q_{\U_{a-1,a}}(x)\, Q_{\U_{n-a-1,n-a}}(x)\right),\]
    as desired. 
\end{proof}

With this formula at hand, it is straightforward to compute $Q_{\M}(x)$ and $Y_{\M}(x)$ for very large matroids $\M$ of corank $2$. For all such matroids, up to cardinality $20$, Conjecture~\ref{conj:xie-zhang} is true. However, starting at $21$ elements counterexamples appear. The following result implies Theorem~\ref{thm:counterexample-main}.

\begin{theorem}\label{thm:counterexample-body}
    The corank $2$ matroid induced by partitioning the set $\{1,\ldots,21\}$ into six parts of sizes $(4,4,4,3,3,3)$ is a counterexample to Conjecture~\ref{conj:xie-zhang}.
\end{theorem}

\begin{proof}
    Computing the inverse Kazhdan--Lusztig polynomial via the preceding theorem results in
    \begin{multline*} Q_{\M}(x) = 71162 \, x^{9} + 232662 \, x^{8} + 350404 \, x^{7} + 323646 \, x^{6}\\ + 215169 \, x^{5} + 106659 \, x^{4} + 39217 \, x^{3} + 10323 \, x^{2} + 1790 \, x + 163.
    \end{multline*}
    After normalizing, we obtain
    \begin{multline*} \mathcal{B}(Q_{\M})(x)= 71162 \, x^{9} + 2093958 \, x^{8} + 12614544 \, x^{7} + 27186264 \, x^{6}\\ + 27111294 \, x^{5} + 13439034 \, x^{4} + 3294228 \, x^{3} + 371628 \, x^{2} + 16110 \, x + 163.
    \end{multline*}
    According to SageMath, this polynomial is not real-rooted. It possesses two complex conjugate zeros near $-1.0297\pm 0.1097\,i$. 
\end{proof}

\begin{remark}
    We were unable to find counterexamples for the logarithmic concavity of $Q_{\M}(x)$ and $Y_{\M}(x)$ using matroids of the form described above. We have verified with the help of a computer that Conjecture~\ref{conj:four-conjectures} \eqref{it:conj-q} and \eqref{it:conj-y} hold for all corank $2$ matroids of cardinality less than or equal to $35$.
\end{remark}

\bibliographystyle{amsalpha}
\bibliography{references}

\end{document}